\documentclass[twoside,reqno]{amsart}

\usepackage{amsmath,amsfonts,amsthm}
\usepackage[colorlinks=true,linkcolor=blue,citecolor=blue,%
filecolor=blue,urlcolor=blue,pageanchor=true,plainpages=false,%
linktocpage]{hyperref}

\numberwithin{equation}{section}
\newtheorem{thm}{Theorem}[section]
\newtheorem{prop}[thm]{Proposition}
\newtheorem{cor}[thm]{Corollary}
\newtheorem{lem}[thm]{Lemma}
\theoremstyle{definition}
\newtheorem{defn}[thm]{Definition}
\newtheorem{rem}[thm]{Remark}

\newcommand{\Z}{\mathbb{Z}}
\newcommand{\R}{\mathbb{R}}

\newcommand{\idx}[1]{i\left(#1\right)}

\newcommand{\au}[1]{\textsc{#1}}
\newcommand{\titleart}[1]{\textrm{#1}}
\newcommand{\jour}[1]{\textit{#1}}
\newcommand{\volart}[1]{\textbf{#1}}
\newcommand{\no}[1]{\textit{no.} {#1}}

\renewcommand{\rho}{\varrho}
\renewcommand{\theta}{\vartheta}
\begin{document}

%TOPMATTER

\title[The 
eigenvalues of the $p$-Laplace operator]{On 
the dependence on $p$ of the variational
eigenvalues of the $p$-Laplace operator}

\author{Marco Degiovanni \and Marco Marzocchi}
\address{Dipartimento di Matematica e Fisica\\
         Universit\`a Cattolica del Sacro Cuore\\
         Via dei Musei 41\\
         25121 Bre\-scia, Italy}
\email{marco.degiovanni@unicatt.it, marco.marzocchi@unicatt.it}
\thanks{The research of the authors was partially supported by 
        Gruppo Nazionale per l'Analisi Matematica,
        la Probabilit\`a e le loro Applicazioni (INdAM)}
				


%END TOPMATTER

%--------------------------------------------------------------------

%
\begin{abstract}
We study the behavior of the variational eigenvalues of the 
$p$-Laplace operator, with homogeneous Dirichlet boundary
condition, when $p$ is varying.
After introducing an auxiliary problem, we characterize the
continuity answering, in particular, a question
raised in~\cite{lindqvist1993}.
\end{abstract}
\maketitle %AMSART

%--------------------------------------------------------------------

\section{Introduction}
Let $\Omega$ be a connected and bounded open subset of $\R^N$
and let $1<p<\infty$.
The study of the nonlinear eigenvalue problem
\begin{equation}
\label{eq:main}
\begin{cases}
- \Delta_p u 
= \lambda |u|^{p-2}u
&\qquad\text{in $\Omega$}\,,\\
\noalign{\medskip}
u=0
&\qquad\text{on $\partial\Omega$}\,,
\end{cases}
\end{equation}
namely
\[
\begin{cases}
u\in W^{1,p}_0(\Omega)\,,\\
\noalign{\medskip}
\displaystyle{
\int_\Omega |\nabla u|^{p-2} \nabla u\cdot \nabla v\,dx =
\lambda \int_\Omega |u|^{p-2}u v\,dx}
\qquad\forall v\in W^{1,p}_0(\Omega)\,,
\end{cases}
\]
has been the object of several papers, starting 
from~\cite{lindqvist1990}, where it has been proved that 
the first eigenvalue is simple and is the unique
eigenvalue which admits a positive eigenfunction.
Alternative proofs and more general equations have been
the object of the subsequent 
papers~\cite{belloni_kawohl2002, brasco_franzina2012, 
kawohl_lindqvist2006, kawohl_lucia_prashanth2007,
lucia_prashanth2006, lucia_schuricht2013},
while the existence of a diverging sequence of eigenvalues
has been proved under quite general assumptions 
in~\cite{lucia_schuricht2013, szulkin_willem1999}.
\par
If we denote by $\lambda_p^{(1)}$ the first eigenvalue
of~\eqref{eq:main} and by $u_p$ the associated
positive eigenfunction such that
\[
\int_\Omega u_p^p\,dx = 1\,,
\]
a challenging question concerns the behavior of
$\lambda_p^{(1)}$ and $u_p$ with respect to $p$.
As shown in~\cite{lindqvist1993}, about the dependence 
from the right one has in full generality 
\begin{alignat*}{3}
&\lim_{s\to p^+} \lambda_s^{(1)} = \lambda_p^{(1)}\,,\\
&\lim_{s\to p^+}
\int_\Omega |\nabla u_s - \nabla u_p|^p\,dx = 0\,,
\end{alignat*}
while the ``corresponding'' assertions from the left
\begin{alignat*}{3}
&\lim_{s\to p^-} \lambda_s^{(1)} = \lambda_p^{(1)}\,,\\
&\lim_{s\to p^-}
\int_\Omega |\nabla u_s - \nabla u_p|^s\,dx = 0
\end{alignat*}
are true under some further assumption about $\partial\Omega$.
A counterexample in the same paper~\cite{lindqvist1993}
shows that otherwise in general they are false.
\par
A related question concerns the equivalence, without
any assumption on $\partial\Omega$, between the two 
assertions.
In~\cite{lindqvist1993} it is proved that, if
\[
\lim_{s\to p^-}
\int_\Omega |\nabla u_s - \nabla u_p|^s\,dx = 0\,,
\]
then
\[
\lim_{s\to p^-} \lambda_s^{(1)} = \lambda_p^{(1)}\,,
\]
while the converse is proposed as an open problem.
Subsequent papers have considered more general situations 
(see~\cite{elkhalil_elmanouni_ouanan2006,
elkhalil_lindqvist_touzani2004}),
but the previous question seems to be still unsolved
(see also~\cite{lindqvist2008}).
\par
The main purpose of this paper is to introduce an auxiliary
problem which allows to describe the behaviour
of $\lambda_s^{(1)}$ and $u_s$ as $s\to p^-$
(see the next Theorem~\ref{thm:limlambda}).
Then in Theorem~\ref{thm:char} we provide several equivalent
characterizations of the fact that
\[
\lim_{s\to p^-} \lambda_s^{(1)} = \lambda_p^{(1)}\,.
\]
In particular, in Corollary~\ref{cor:conv} we give a positive
answer to the mentioned open problem.
\par
We also consider the dependence on $s$ of the full sequence
$(\lambda_s^{(m)})$ of the variational eigenvalues, defined 
according to some topological index $i$.
In particular, in Corollary~\ref{cor:high} we prove that
\[
\lim_{s\to p} \lambda_s^{(m)} =
\lambda_p^{(m)} 
\qquad\forall m\geq 1
\]
if and only if
\[
\lim_{s\to p^-} \lambda_s^{(1)} =
\lambda_p^{(1)} \,.
\]
The convergence of $\lambda_s^{(m)}$ has been already
studied in~\cite{champion_depascale2007}, under
the $\Gamma$-convergence of the associated functionals.
More specifically, in~\cite{parini2011} it has been proved
the continuity of $\lambda_s^{(m)}$ with respect to $s$,
provided that $\partial\Omega$ is smooth enough.

%--------------------------------------------------------------------

\section{The first eigenvalue with respect to a larger space}
\label{sect:large}
Throughout the paper, $\Omega$ will denote a bounded
and open subset of~$\R^N$.
No assumption will be imposed \emph{a priori} about the 
regularity of $\partial\Omega$.
We will also denote by $\mathcal{L}^N$ the Lebesgue 
measure in $\R^N$.
\par
If $u\in W^{1,p}(\Omega)$, the condition 
``\emph{$u=0$ on $\partial\Omega$}'' is usually expressed
by saying that $u\in W^{1,p}_0(\Omega)$.
If $\partial\Omega$ is smooth enough, this is perfectly
reasonable; if not, other (nonequivalent) formulations
can be proposed.
In the line of the approach of~\cite{lindqvist1993}, 
if $1<p<\infty$ we set
\[
W^{1,p_-}_0(\Omega) =
W^{1,p}(\Omega) \cap\left(\bigcap_{1<s<p} W^{1,s}_0(\Omega)\right)
= \bigcap_{1<s<p} \left(W^{1,p}(\Omega)\cap W^{1,s}_0(\Omega)
\right)\,.
\]
\begin{prop}
\label{prop:wp-}
The following facts hold:
\begin{itemize}
\item[$(a)$]
$W^{1,p_-}_0(\Omega)$ is a closed vector subspace
of $W^{1,p}(\Omega)$ satisfying
\[
W^{1,p}_0(\Omega) \subseteq W^{1,p_-}_0(\Omega)\,;
\]
\item[$(b)$]
for every $u\in W^{1,p_-}_0(\Omega)$, the function
\[
\hat{u} = 
\begin{cases}
u &\qquad\text{on $\Omega$}\,,\\
\noalign{\medskip}
0 &\qquad\text{on $\R^N\setminus\Omega$}
\end{cases}
\]
belongs to $W^{1,p}(\R^N)$; in particular,
\[
\left(\int_\Omega |\nabla u|^p\,dx\right)^{1/p}
\]
is a norm on $W^{1,p_-}_0(\Omega)$ equivalent to the one
induced by $W^{1,p}(\Omega)$;
\item[$(c)$]
if $p<N$, we have 
$W^{1,p_-}_0(\Omega)\subseteq L^{p^*}(\Omega)$ and
\begin{multline*}
\null\qquad\qquad
\inf\left\{\frac{\int_\Omega |\nabla u|^p\,dx}{
\left(\int_\Omega |u|^{p^*}\,dx\right)^{p/{p^*}}}:\,\,
u\in W^{1,p}_0(\Omega)\setminus\{0\}\right\} \\
=
\inf\left\{\frac{\int_\Omega |\nabla u|^p\,dx}{
\left(\int_\Omega |u|^{p^*}\,dx\right)^{p/{p^*}}}:\,\,
u\in W^{1,p_-}_0(\Omega)\setminus\{0\}\right\} \,;
\qquad\qquad\null
\end{multline*}
\item[$(d)$]
if $p>N$, we have $W^{1,p_-}_0(\Omega)=W^{1,p}_0(\Omega)$;
\item[$(e)$]
if $\Omega$ has the segment property, 
we have $W^{1,p_-}_0(\Omega)=W^{1,p}_0(\Omega)$ for any $p$.
\end{itemize}
\end{prop}
\begin{proof}
Since $W^{1,p}(\Omega)\cap W^{1,s}_0(\Omega)$ is a
closed vector subspace of $W^{1,p}(\Omega)$ containing
$W^{1,p}_0(\Omega)$, assertion~$(a)$ follows.
\par
If $u\in W^{1,p_-}_0(\Omega)$, the function
\[
\hat{u} = 
\begin{cases}
u &\qquad\text{on $\Omega$}\,,\\
\noalign{\medskip}
0 &\qquad\text{on $\R^N\setminus\Omega$}
\end{cases}
\]
belongs to $W^{1,s}(\R^N)$ for any $s<p$ and
\[
- \int_{\R^N} \hat{u} D_jv\,dx =
\int_\Omega D_ju v\,dx
\qquad\forall v\in C^1_c(\R^N)\,.
\]
It follows $\hat{u}\in W^{1,p}(\R^N)$, whence assertion $(b)$.
\par
If $p<N$, let $U$ be a bounded open subset of $\R^N$
with $\overline{\Omega}\subseteq U$
and let $u\in W^{1,p_-}_0(\Omega)$.
Then $\hat{u}\in W^{1,p}_0(U)\subseteq L^{p^*}(U)$ 
(see e.g.~\cite[Lemma~9.5]{brezis2011}) and
\[
\frac{\int_U |\nabla \hat{u}|^p\,dx}{
\left(\int_U |\hat{u}|^{p^*}\,dx\right)^{p/{p^*}}} =
\frac{\int_\Omega |\nabla u|^p\,dx}{
\left(\int_\Omega |u|^{p^*}\,dx\right)^{p/{p^*}}}\,.
\]
Assertion~$(c)$ follows from the fact that 
\[
\inf\left\{\frac{\int_\Omega |\nabla v|^p\,dx}{
\left(\int_\Omega |v|^{p^*}\,dx\right)^{p/{p^*}}}:\,\,
v\in W^{1,p}_0(\Omega)\setminus\{0\}\right\} 
\]
is independent of $\Omega$ (see e.g.~\cite{talenti1976}).
\par
If $p>N$ and $u\in W^{1,p_-}_0(\Omega)$, we have
$\hat{u}\in C(\overline{\Omega})\cap W^{1,p}(\Omega)$
with $\hat{u}=0$ on $\partial\Omega$.
It follows that $\hat{u}\in W^{1,p}_0(\Omega)$
(see e.g.~\cite[Theorem~9.17]{brezis2011}, where the proof
of $(i)\Rightarrow(ii)$ does not use the regularity
of $\partial\Omega$), whence assertion~$(d)$.
\par
Assertion $(e)$ is taken
from~\cite[Theorem~2.1]{elkhalil_lindqvist_touzani2004}.
\end{proof}
Let us point out that the counterexample
in~\cite{lindqvist1993} shows that  
$W^{1,p_-}_0(\Omega)$ can be strictly larger than
$W^{1,p}_0(\Omega)$ in the case $1<p\leq N$
(see also the next Remark~\ref{rem:diff}).
\par
Now we set
\begin{alignat*}{3}
&\lambda_p^{(1)} &&= 
\inf\left\{\frac{\int_\Omega |\nabla u|^p\,dx}{
\int_\Omega |u|^p\,dx}:\,\,
u\in W^{1,p}_0(\Omega)\setminus\{0\}\right\}\,,\\
&\underline{\lambda}_p^{(1)} &&= 
\inf\left\{\frac{\int_\Omega |\nabla u|^p\,dx}{
\int_\Omega |u|^p\,dx}:\,\,
u\in W^{1,p_-}_0(\Omega)\setminus\{0\}\right\}\,.
\end{alignat*}
It is easily seen that also the infimum defining
$\underline{\lambda}_p^{(1)}$ is achieved and 
we clearly have
\[
0 < \underline{\lambda}_p^{(1)} \leq
\lambda_p^{(1)}\,.
\]
More precisely, there exists $v\in W^{1,p_-}_0(\Omega)$
such that
\[
v\geq 0\quad\text{a.e. in $\Omega$}\,,\qquad
\int_\Omega v^p\,dx = 1\,,\qquad
\int_\Omega |\nabla v|^p\,dx =  
\underline{\lambda}_p^{(1)}\,.
\]
According to~\cite{lindqvist1990, lindqvist1993},
if $\Omega$ is connected we also denote by $u_p$ 
the positive eigenfunction in $W^{1,p}_0(\Omega)$ 
associated with $\lambda_p^{(1)}$ such that
\[
\int_\Omega u_p^p\,dx = 1\,.
\]
In the next result we will see that something similar 
can be done with respect to the space $W^{1,p_-}_0(\Omega)$.
\begin{thm}
\label{thm:lambda1}
If $\Omega$ is connected, the following facts hold:
\begin{itemize}
\item[$(a)$]
there exists one and only one 
$\underline{u}_p\in W^{1,p_-}_0(\Omega)$ such that
\[
\text{$\underline{u}_p\geq 0$ a.e. in $\Omega$}\,,\qquad
\int_\Omega \underline{u}_p^p\,dx = 1\,,
\qquad
\int_\Omega |\nabla \underline{u}_p|^p\,dx
= \underline{\lambda}_p^{(1)} \,;
\]
moreover, we have 
$\underline{u}_p\in L^\infty(\Omega)\cap C^1(\Omega)$
and $\underline{u}_p>0$ in $\Omega$;
\item[$(b)$]
the set of $u$'s in $W^{1,p_-}_0(\Omega)$ such that
\[
\int_\Omega |\nabla u|^{p-2}\nabla u\cdot\nabla v\,dx
= 	\underline{\lambda}_p^{(1)}
\int_\Omega |u|^{p-2}u v\,dx
 \qquad\forall v\in W^{1,p_-}_0(\Omega)
\]
is a vector subspace of $W^{1,p_-}_0(\Omega)$ of
dimension $1$;
\item[$(c)$]
if $\lambda\in\R$ and 
$u\in W^{1,p_-}_0(\Omega)\setminus\{0\}$ satisfy
\[
\begin{cases}
\text{$u\geq 0$ a.e. in $\Omega$} \,,\\
\noalign{\medskip}
\displaystyle{
\int_\Omega |\nabla u|^{p-2}\nabla u\cdot\nabla v\,dx
= \lambda
\int_\Omega u^{p-1} v\,dx}
\qquad\forall v\in W^{1,p_-}_0(\Omega) \,,
\end{cases}
\] 
then $\lambda=\underline{\lambda}_p^{(1)}$
and $u=t\,\underline{u}_p$ for some $t>0$.
\end{itemize}
\end{thm}
Since the proof follows the same lines
of~\cite{lindqvist1990}, we pospone it to the
Appendix.

%--------------------------------------------------------------------

\section{Behavior from the left of the first eigenvalue}
\label{sect:left}
\begin{prop}
\label{prop:ineq}
If $1<s<p<\infty$, it holds
\[
s\,\left(\underline{\lambda}_s^{(1)}\right)^{1/s} \leq
s\,\left(\lambda_s^{(1)}\right)^{1/s} < 
p\,\left(\underline{\lambda}_p^{(1)}\right)^{1/p} \leq
p\,\left(\lambda_p^{(1)}\right)^{1/p} \,. 
\]
\end{prop}
\begin{proof}
Let $s<t<p$.
For every $u\in W^{1,p_-}_0(\Omega)\cap L^\infty(\Omega)$, we have
$|u|^{p/t}\in W^{1,t}_0(\Omega)$, hence
\begin{multline*}
\lambda_t^{(1)} \int_\Omega |u|^p\,dx \leq
\int_\Omega \left|\nabla |u|^{p/t}\right|^t\,dx
=
\left(\frac{p}{t}\right)^t\int_\Omega
|u|^{p-t}|\nabla u|^t\,dx \\
\leq
\left(\frac{p}{t}\right)^t
\left(\int_\Omega |u|^p\,dx\right)^{\frac{p-t}{p}}
\left(\int_\Omega |\nabla u|^p\,dx
\right)^{\frac{t}{p}} \,.
\end{multline*}
It follows
\[
\lambda_t^{(1)} 
\left(\int_\Omega |u|^p\,dx
\right)^{\frac{t}{p}}
\leq
\left(\frac{p}{t}\right)^t
\left(\int_\Omega |\nabla u|^p\,dx
\right)^{\frac{t}{p}} 
\qquad\forall u\in W^{1,p_-}_0(\Omega)\cap L^\infty(\Omega)\,.
\]
For every $u\in W^{1,p_-}_0(\Omega)$, the function
\[
v_k=\max\{\min\{u,-k\},k\}
\]
belongs to $W^{1,p_-}_0(\Omega)\cap L^\infty(\Omega)$
and is convergent to $u$ in $W^{1,p}(\Omega)$.
It follows
\[
\lambda_t^{(1)} 
\left(\int_\Omega |u|^p\,dx
\right)^{\frac{t}{p}}
\leq
\left(\frac{p}{t}\right)^t
\left(\int_\Omega |\nabla u|^p\,dx
\right)^{\frac{t}{p}} 
\qquad\forall u\in W^{1,p_-}_0(\Omega)\,,
\]
whence
\[
\lambda_t^{(1)} 
\leq
\left(\frac{p}{t}\right)^t
\left(\underline{\lambda}_p^{(1)}\right)^{\frac{t}{p}} \,.
\]
On the other hand, in~\cite[Remark~p.~204]{lindqvist1993}
it is proved that
\[
s\,\left(\lambda_s^{(1)}\right)^{1/s} <
t\,\left(\lambda_t^{(1)}\right)^{1/t} 
\]
and the argument does not require $\Omega$ to be connected.
We conclude that
\[
s\,\left(\lambda_s^{(1)}\right)^{1/s} < 
p\,\left(\underline{\lambda}_p^{(1)}\right)^{1/p} \,. 
\]
For the sake of completeness we have also recalled
the other (easy) inequalities.
\end{proof}
Now we can prove the main result of this section.
\begin{thm}
\label{thm:limlambda}
If $1<p<\infty$, it holds
\[
\lim_{s\to p^-} \underline{\lambda}_s^{(1)} =
\lim_{s\to p^-} \lambda_s^{(1)} =
\underline{\lambda}_p^{(1)} \,.
\]
If $\Omega$ is connected, we also have
\[
\lim_{s\to p^-}\,
\int_\Omega |\nabla\underline{u}_s - \nabla\underline{u}_p|^s\,dx
=
\lim_{s\to p^-}\,
\int_\Omega |\nabla u_s - \nabla\underline{u}_p|^s\,dx = 0\,.
\]
\end{thm}
\begin{proof}
By Proposition~\ref{prop:ineq} it is clear that
\[
\lim_{s\to p^-} \underline{\lambda}_s^{(1)} \leq
\lim_{s\to p^-} \lambda_s^{(1)} \leq
\underline{\lambda}_p^{(1)} \,.
\]
Let $(p_k)$ be a sequence strictly increasing to $p$,
let $v_k\in W^{1,({p_k})_-}_0(\Omega)$ with
\[
v_k\geq 0\quad\text{a.e. in $\Omega$}\,,\qquad
\int_\Omega v_k^{p_k}\,dx = 1\,,\qquad
\int_\Omega |\nabla v_k|^{p_k}\,dx =  
\underline{\lambda}_{p_k}^{(1)}
\]
and let $1<t<p$.
In particular, it holds
\[
\lim_{k\to\infty}\,
\int_\Omega |\nabla v_k|^{p_k}\,dx <+\infty\,.
\]
Up to a subsequence, we have $p_k>t$ and
$(v_k)$ is
weakly convergent to some $u$ in 
$W^{1,t}_0(\Omega)$.
Since the sequence $(v_k)$ is
eventually bounded in $W^{1,s}_0(\Omega)$ for any
$s<p$, it follows
\[
u\in \bigcap_{1<s<p} W^{1,s}_0(\Omega)\,.
\]
Moreover, it holds
\[
\text{$u\geq 0$ a.e. in $\Omega$}\,,\qquad
\int_\Omega u^p\,dx = 1
\]
and, for every $s<p$,
\begin{alignat*}{3}
&\int_\Omega |\nabla u|^s\,dx &&\leq
\liminf_{k\to\infty}
\int_\Omega |\nabla v_k|^s\,dx \leq
\liminf_{k\to\infty}
\left[\mathcal{L}^N(\Omega)^{1 - \frac{s}{p_k}}\,
\left(\int_\Omega |\nabla v_k|^{p_k}\,dx
\right)^{\frac{s}{p_k}}\right] \\
&&&=
\lim_{k\to\infty}
\left[\mathcal{L}^N(\Omega)^{1 - \frac{s}{p_k}}\,
\left(\underline{\lambda}_{p_k}^{(1)}\right)^{\frac{s}{p_k}}
\right] =
\mathcal{L}^N(\Omega)^{1 - \frac{s}{p}}\,
\left(\lim_{k\to\infty}
\underline{\lambda}_{p_k}^{(1)}\right)^{\frac{s}{p}} \,.
\end{alignat*}
By the arbitrariness of $s$, we infer that
$u\in W^{1,p}(\Omega)$, hence $u\in W^{1,p_-}_0(\Omega)$, with
\[
\underline{\lambda}_{p}^{(1)} \leq
\int_\Omega |\nabla u|^p\,dx \leq
\lim_{k\to\infty}
\underline{\lambda}_{p_k}^{(1)} \,.
\]
It follows
\[
\lim_{s\to p^-} \underline{\lambda}_s^{(1)} =
\lim_{s\to p^-} \lambda_s^{(1)} =
\underline{\lambda}_p^{(1)} =
\int_\Omega |\nabla u|^p\,dx\,.
\]
Now assume that $\Omega$ is connected.
From~$(a)$ of Theorem~\ref{thm:lambda1}, we infer that
$v_k=\underline{u}_{p_k}$, $u=\underline{u}_p$ and 
\[
\lim_{s\to p^-} \underline{u}_s = \underline{u}_p
\qquad\text{weakly in $W^{1,t}_0(\Omega)$ 
for any $t<p$}\,.
\]
In particular, it holds
\[
\lim_{s\to p^-} 
\int_\Omega \left|\frac{\underline{u}_s +
\underline{u}_p}{2}\right|^s\,dx = 1\,,
\]
whence
\[
\liminf_{s\to p^-}
\int_\Omega \left|\frac{\nabla \underline{u}_s +
\nabla\underline{u}_p}{2}\right|^s\,dx \geq
\underline{\lambda}_p\,.
\]
If $2<s<p$, Clarkson's inequality yields
\[
\int_\Omega \left|\frac{\nabla \underline{u}_s +
\nabla\underline{u}_p}{2}\right|^s\,dx +
\int_\Omega \left|\frac{\nabla \underline{u}_s -
\nabla\underline{u}_p}{2}\right|^s\,dx \leq
\frac{1}{2}
\int_\Omega \left|\nabla \underline{u}_s\right|^s\,dx +
\frac{1}{2}
\int_\Omega \left|\nabla \underline{u}_p\right|^s\,dx\,,
\]
whence
\[
\lim_{s\to p^-} 
\int_\Omega \left|\nabla \underline{u}_s -
\nabla\underline{u}_p\right|^s\,dx = 0
\qquad\text{if $p>2$}\,.
\]
When $1<s<p\leq 2$, Clarkson's inequality becomes
\begin{multline*}
\left(\int_\Omega \left|\frac{\nabla \underline{u}_s +
\nabla\underline{u}_p}{2}\right|^s\,dx\right)^{\frac{1}{s-1}} +
\left(\int_\Omega \left|\frac{\nabla \underline{u}_s -
\nabla\underline{u}_p}{2}\right|^s\,dx\right)^{\frac{1}{s-1}} 
\\ \leq
\left(\frac{1}{2}
\int_\Omega \left|\nabla \underline{u}_s\right|^s\,dx +
\frac{1}{2}
\int_\Omega \left|\nabla \underline{u}_p\right|^s\,dx
\right)^{\frac{1}{s-1}}\,,
\end{multline*}
but the argument is the same.
\par
Finally, in a similar way one can prove that
\[
\lim_{s\to p^-} 
\int_\Omega \left|\nabla u_s -
\nabla\underline{u}_p\right|^s\,dx = 0\,.
\]
\end{proof}
\begin{rem}
\label{rem:diff}
Since the paper~\cite{lindqvist1993} contains a counterexample with 
\[
\lim_{s\to p^-} \lambda_s^{(1)} < \lambda_p^{(1)} \,,
\]
in that case we have 
$\underline{\lambda}_p^{(1)} < \lambda_p^{(1)}$, hence
$W^{1,p_-}_0(\Omega)\neq W^{1,p}_0(\Omega)$.
\end{rem}
Now we aim also to describe the behavior as $s\to p^-$ in the terms
of the variational convergence 
of~\cite{attouch1984, dalmaso1993}.
\begin{defn}
Let $X$ be a metrizable topological space,
$f:X\rightarrow[-\infty,+\infty]$ a function and
let $(f_h)$ be a sequence of functions from $X$ 
to~$[-\infty,+\infty]$.
According to~\cite[Proposition~8.1]{dalmaso1993}, we say that
$(f_h)$ is $\Gamma$-convergent to $f$ and we write
\[
\Gamma-\lim_{h\to\infty} f_h = f\,, 
\]
if the following facts hold:
\item[$(a)$]
for every $u\in X$ and every sequence $(u_h)$ converging to 
$u$ in $X$ it holds
\[
\liminf_{h\to\infty} f_h(u_h) \geq f(u)\,;
\]
\item[$(b)$]
for every $u\in X$ there exists a sequence $(u_h)$
converging to $u$ in $X$ such that
\[
\lim_{h\to\infty} f_h(u_h) = f(u)\,.
\]
\end{defn}
If $1<p<\infty$, we define two functionals
$\mathcal{E}_p, 
\underline{\mathcal{E}}_p:L^1_{loc}(\Omega)\rightarrow[0,+\infty]$
as
\begin{alignat*}{3}
&\mathcal{E}_p(u) &&=
\begin{cases}
\displaystyle{
\left(\int_\Omega|\nabla u|^p\,dx\right)^{1/p}}
&\qquad\text{if $u\in W^{1,p}_0(\Omega)$}\,,\\
\noalign{\medskip}
+\infty
&\qquad\text{otherwise}\,,
\end{cases}\\
&\underline{\mathcal{E}}_p(u) &&=
\begin{cases}
\displaystyle{
\left(\int_\Omega|\nabla u|^p\,dx\right)^{1/p}}
&\qquad\text{if $u\in W^{1,p_-}_0(\Omega)$}\,,\\
\noalign{\medskip}
+\infty
&\qquad\text{otherwise}\,.
\end{cases}
\end{alignat*}
\begin{thm}
\label{thm:gammal}
For every sequence $(p_h)$ strictly increasing to $p$, with
$1<p<\infty$, it holds
\[
\Gamma-\lim_{h\to\infty} \underline{\mathcal{E}}_{p_h} = 
\Gamma-\lim_{h\to\infty} \mathcal{E}_{p_h} = 
\underline{\mathcal{E}}_p\,.
\]
\end{thm}
\begin{proof}
Define, whenever $1<s<\infty$,
$f_s, 
\underline{f}_s:L^1_{loc}(\Omega)\rightarrow[0,+\infty]$
as
\[
f_s(u) = \mathcal{L}^N(\Omega)^{-1/s}\,\mathcal{E}_s(u)\,,\qquad
\underline{f}_s(u) = \mathcal{L}^N(\Omega)^{-1/s}\,
\underline{\mathcal{E}}_s(u)\,.
\]
Then $f_s, \underline{f}_s$ are lower semicontinuous
and the sequences $(f_{p_h})$, $(\underline{f}_{p_h})$ are
both increasing and pointwise convergent
to $\underline{f}_p$.
From~\cite[Proposition~5.4]{dalmaso1993} we infer that
\[
\Gamma-\lim_{h\to\infty} \underline{f}_{p_h} = 
\Gamma-\lim_{h\to\infty} f_{p_h} = 
\underline{f}_p
\]
and the assertion easily follows.
\end{proof}
%

%--------------------------------------------------------------------

\section{Some characterizations}
\label{sect:char}
Without imposing any assumption on $\partial\Omega$, we aim
to characterize the fact that
\[
\lim_{s\to p^-} \lambda_s^{(1)} = \lambda_p^{(1)} \,.
\]
\begin{thm}
\label{thm:char}
If $1<p<\infty$ and $\Omega$ is connected, the following 
facts are equivalent:
\begin{itemize}
\item[$(a)$]
$\lim\limits_{s\to p^-} \lambda_s^{(1)} = \lambda_p^{(1)}$;
\item[$(b)$]
for every sequence $(p_h)$ strictly increasing to $p$, it holds
\[
\Gamma-\lim_{h\to\infty} \mathcal{E}_{p_h} = 
\mathcal{E}_p\,;
\]
\item[$(c)$]
$W^{1,p_-}_0(\Omega)=W^{1,p}_0(\Omega)$;
\item[$(d)$]
$\underline{\lambda}_p^{(1)} =\lambda_p^{(1)}$;
\item[$(e)$]
$\underline{u}_p =u_p$;
\item[$(f)$]
$\underline{u}_p \in W^{1,p}_0(\Omega)$;
\item[$(g)$]
the solution $u$ of
\[
\begin{cases}
u\in W^{1,p_-}_0(\Omega)\,,\\
\noalign{\medskip}
\displaystyle{
\int_\Omega |\nabla u|^{p-2}\nabla u\cdot\nabla v\,dx =
\int_\Omega v\,dx}
\qquad\forall v\in W^{1,p_-}_0(\Omega)\,,
\end{cases}
\]
given by the next Theorem~\ref{thm:mon}, belongs to $W^{1,p}_0(\Omega)$.
\end{itemize}
\end{thm}
\begin{proof}
By Theorem~\ref{thm:limlambda} it is clear that
$(a)\Leftrightarrow(d)$, while we have $(b)\Leftrightarrow(c)$
by Theorem~\ref{thm:gammal}.
Now we consider the assertions from $(c)$ to~$(g)$.
\par
It is clear that $(c)\Rightarrow(d)$.
If $\underline{\lambda}_p^{(1)} =\lambda_p^{(1)}$,
we have that 
$u_p \in W^{1,p}_0(\Omega)\subseteq W^{1,p_-}_0(\Omega)$
satisfies
\[
\text{$u_p\geq 0$ a.e. in $\Omega$}\,,\qquad
\int_\Omega u_p^p\,dx = 1\,,
\qquad
\int_\Omega |\nabla u_p|^p\,dx
= \underline{\lambda}_p^{(1)} \,.
\]
From $(a)$ of Theorem~\ref{thm:lambda1} we infer that
$u_p=\underline{u}_p$.
Therefore $(d)\Rightarrow(e)$.
\par
Of course, $(e)\Rightarrow(f)$.
If $\underline{u}_p \in W^{1,p}_0(\Omega)$, let 
\[
z_k = \min\left\{
\underline{\lambda}_p^{(1)}(k \underline{u}_p)^{p-1},1\right\}
\]
and let $w_k\in W^{1,p_-}_0(\Omega)$ be the solution of
\[
\int_\Omega |\nabla w_k|^{p-2}\nabla w_k\cdot\nabla v\,dx
= \int_\Omega z_k v\,dx
\qquad\forall v\in W^{1,p_-}_0(\Omega)
\]
according to Theorem~\ref{thm:mon}.
Since $0\leq z_k\leq 
\underline{\lambda}_p^{(1)}(k \underline{u}_p)^{p-1}$
a.e. in $\Omega$, we have
$0\leq w_k\leq k \underline{u}_p$ a.e. in $\Omega$.
From $w_k\in W^{1,p}(\Omega)$ and 
$k \underline{u}_p\in W^{1,p}_0(\Omega)$, we infer that
$w_k\in W^{1,p}_0(\Omega)$.
\par
Since $(z_k)$ is convergent to $1$ in $L^p(\Omega)$, 
we also have
\[
\lim_{k\to\infty}
\int_\Omega |\nabla w_k - \nabla u|^p\,dx = 0\,,
\]
whence $u\in W^{1,p}_0(\Omega)$.
Therefore $(f)\Rightarrow(g)$.
\par
Finally, assume that~$(g)$ holds and let $u$ be
as in assertion~$(g)$.
If $z\in L^\infty(\Omega)$ and $w\in W^{1,p_-}_0(\Omega)$ 
is the solution of
\[
\int_\Omega |\nabla w|^{p-2}\nabla w\cdot\nabla v\,dx
= \int_\Omega z v\,dx
\qquad\forall v\in W^{1,p_-}_0(\Omega) \,,
\]
we have $-M^{p-1}\leq z\leq M^{p-1}$ for some $M>0$, whence
$-Mu \leq w\leq Mu$ a.e. in $\Omega$.
It follows $w\in W^{1,p}_0(\Omega)$.
\par
Now let $w\in W^{1,p_-}_0(\Omega)$.
Let $z\in L^p(\Omega)$ and $Z\in L^p(\Omega;\R^N)$
be such that
\[
\int_\Omega |\nabla w|^{p-2}\nabla w\cdot\nabla v\,dx =
\int_\Omega (zv+Z\cdot\nabla v)\,dx
\qquad\forall v\in W^{1,p}(\Omega)\,.
\]
Then let $(z_k)$ and $(Z_k)$ be two sequences in $C^\infty_c$
converging to $z$ and $Z$, respectively, in~$L^p$.
Since $(z_k-\mathrm{div}Z_k)\in L^\infty(\Omega)$,
there exists $w_k\in W^{1,p}_0(\Omega)$ such that
\[
\int_\Omega |\nabla w_k|^{p-2}\nabla w_k\cdot\nabla v\,dx
= \int_\Omega (z_k-\mathrm{div}Z_k) v\,dx
\qquad\forall v\in W^{1,p_-}_0(\Omega) \,.
\]
Since
\begin{alignat*}{3}
&\int_\Omega |\nabla w|^{p-2}\nabla w\cdot\nabla v\,dx &&=
\int_\Omega (zv+Z\cdot\nabla v)\,dx
&&\qquad\forall v\in W^{1,p_-}_0(\Omega)\,,\\
&\int_\Omega |\nabla w_k|^{p-2}\nabla w_k\cdot\nabla v\,dx
&&= \int_\Omega (z_kv+Z_k\cdot\nabla v)\,dx
&&\qquad\forall v\in W^{1,p_-}_0(\Omega) \,,
\end{alignat*}
it follows
\[
\lim_{k\to\infty}
\int_\Omega |\nabla w_k - \nabla w|^p\,dx = 0\,,
\]
whence $w\in W^{1,p}_0(\Omega)$.
Therefore $(g)\Rightarrow(c)$.
\end{proof}
\begin{rem}
If $\Omega$ is not assumed to be connected, it holds
\[
\begin{array}{ccccc}
(b)&\Leftrightarrow&(c)&\Leftrightarrow&(g)\\
&&\Downarrow\\
(a)&\Leftrightarrow&(d)
\end{array}
\]
In fact the same proof shows that
\[
\begin{array}{ccccc}
(b)&\Leftrightarrow&(c)&\Leftarrow&(g)\\
&&\Downarrow\\
(a)&\Leftrightarrow&(d)
\end{array}
\]
and it is obvious that $(c)\Rightarrow(g)$.
\par
On the other hand, let $U$ be a bounded
open set as in Remark~\ref{rem:diff}, with
$W^{1,p_-}_0(U)\neq W^{1,p}_0(U)$,
and let $\Omega=U\cup B$,
where $B$ is an open ball with 
$\overline{U}\cap \overline{B}=\emptyset$.
Then $W^{1,p_-}_0(\Omega)\neq W^{1,p}_0(\Omega)$,
so that $(b)$, $(c)$ and $(g)$ are false. 
However, if the ball~$B$ is large enough, the first 
eigenvalue associated with $\Omega$
coincides with that associated with $B$, which has the 
segment property, so that assertions $(a)$ and $(d)$ are true.
\end{rem}
\begin{rem}
Let us stress, in Theorem~\ref{thm:char}, the assertion
$(a)\Rightarrow(b)$.
When $\Omega$ is connected, the convergence of the first
eigenvalue implies the $\Gamma$-convergence of the full
functional.
This fact will be on the basis of the next
Corollary~\ref{cor:high}. 
\end{rem}
\begin{cor}
\label{cor:conv}
If $\Omega$ is connected and
\[
\lim_{s\to p^-} \lambda_s^{(1)} = \lambda_p^{(1)}\,,
\]
then it holds
\[
\lim_{s\to p^-}\,
\int_\Omega |\nabla u_s - \nabla u_p|^s\,dx = 0\,.
\]
\end{cor}
\begin{proof}
From Theorem~\ref{thm:char} we infer that $\underline{u}_p =u_p$.
By Theorem~\ref{thm:limlambda} we conclude that
\[
\lim_{s\to p^-}\,
\int_\Omega |\nabla u_s - \nabla u_p|^s\,dx =
\lim_{s\to p^-}\,
\int_\Omega |\nabla u_s - \nabla \underline{u}_p|^s\,dx = 0\,.
\]
\end{proof}
\begin{rem}
The converse of the previous Corollary was known
since a long time (see~\cite[Theorem~3.11]{lindqvist1993}),
while Corollary~\ref{cor:conv} was proposed as an open problem.
Corollary~\ref{cor:conv} also answers a question raised
in~\cite{lindqvist1993} concerning the formulation
of~Lemma~3.12 in that paper.
\end{rem}

%--------------------------------------------------------------------

\section{Behavior from the right of the first eigenvalue}
\label{sect:right}
The next results are essentially contained 
in~\cite{lindqvist1993, degiovanni_marzocchi2014}.
We mention them for the sake of completeness.
\begin{thm}
\label{thm:limlambdar}
If $1<p<\infty$, it holds
\[
\lim_{s\to p^+} \lambda_s^{(1)} =
\lim_{s\to p^+} \underline{\lambda}_s^{(1)} =
\lambda_p^{(1)} \,.
\]
If $\Omega$ is connected, we also have
\[
\lim_{s\to p^+}\,
\int_\Omega |\nabla u_s - \nabla u_p|^p\,dx = 
\lim_{s\to p^+}\,
\int_\Omega |\nabla\underline{u}_s - \nabla u_p|^p\,dx
= 0\,.
\]
\end{thm}
\begin{proof}
The assertions concerning $\lambda_s^{(1)}$ and
$u_s$ are proved 
in~\cite[Theorems~3.5 and~3.6]{lindqvist1993}, but the 
same arguments apply also to 
$\underline{\lambda}_s^{(1)}$ and
$\underline{u}_s$.
\end{proof}
\begin{thm}
\label{thm:gammar}
For every sequence $(p_h)$ strictly decreasing to $p$, with
$1<p<\infty$, it holds
\[
\Gamma-\lim_{h\to\infty} \mathcal{E}_{p_h} = 
\Gamma-\lim_{h\to\infty} \underline{\mathcal{E}}_{p_h} = 
\mathcal{E}_p\,.
\]
\end{thm}
\begin{proof}
The assertion concerning $\mathcal{E}_{p_h}$
is proved in~\cite[Theorem~5.3]{degiovanni_marzocchi2014}
when $1<p<N$, but the same argument applies to the other
cases. 
\end{proof}

%--------------------------------------------------------------------

\section{Behavior of higher eigenvalues}
\label{sect:high}
Let $i$ be an index with the following properties:
\begin{itemize}
\item[$(i)$]
$i(K)$ is an integer greater or equal than $1$ and is defined 
whenever $K$ is a nonempty, compact and symmetric subset of a 
topological vector space such that $0\not\in K$;
\item[$(ii)$]
if $X$ is a topological vector space and 
$K\subseteq X\setminus\{0\}$ is compact, symmetric and nonempty, 
then there exists an open subset $U$ of $X\setminus\{0\}$ 
such that $K\subseteq U$ and
\[
~\qquad\quad
\idx{\widehat{K}} \leq \idx{K}
\quad\text{for any compact, symmetric and nonempty 
$\widehat{K}\subseteq U$}\,;
\]
\item[$(iii)$]
if $X, Y$ are two topological vector spaces, 
$K\subseteq X\setminus\{0\}$ is compact, symmetric and nonempty 
and $\pi:K\rightarrow Y\setminus\{0\}$ is continuous and 
odd, we have
\[
\idx{\pi(K)} \geq \idx{K}\,;
\]
\item[$(iv)$]
if $X$ is a normed space with
$1\leq \dim X <\infty$, we have
\[
i\left(\left\{u\in X:\,\,\|u\|=1\right\}\right) =
\dim X\,.
\]
\end{itemize} 
Well known examples are the Krasnosel'ski\u{\i} genus
(see e.g.~\cite{krasnoselskii1964, rabinowitz1986, willem1996}) 
and the $\Z_2$-cohomo\-logical index 
(see~\cite{fadell_rabinowitz1977, fadell_rabinowitz1978}).
More general examples are contained in~\cite{bartsch1993}.
\par
If $1<p<\infty$, we consider
\begin{alignat*}{3}
&M = &&\left\{u\in W^{1,p}_0(\Omega):\,\,
\int_\Omega |u|^p\,dx=1\right\}\,,\\
&\underline{M} = &&\left\{u\in W^{1,p_-}_0(\Omega):\,\,
\int_\Omega |u|^p\,dx=1\right\}\,,
\end{alignat*}
endowed with the $W^{1,p}(\Omega)$-topology, and
we define for every $m\geq 1$ the variational eigenvalues 
of the $p$-Laplace operator as
\begin{alignat*}{3}
&\lambda_p^{(m)} &&= 
\inf\biggl\{\max_{u\in K}\int_\Omega |\nabla u|^p\,dx:\,\,
\text{$K$ is a nonempty, compact}
\biggr. \\ \biggl. &&&\qquad\qquad\qquad\qquad\qquad
\text{and symmetric subset of $M$ with 
$i(K)\geq m$}\biggr\}\,,\\
&\underline{\lambda}_p^{(m)} &&= 
\inf\biggl\{\max_{u\in K}\int_\Omega |\nabla u|^p\,dx:\,\,
\text{$K$ is a nonempty, compact}
\biggr. \\ \biggl. &&&\qquad\qquad\qquad\qquad\qquad
\text{and symmetric subset of $\underline{M}$ with 
$i(K)\geq m$}\biggr\}\,.
\end{alignat*}
It is easily seen that the new definitions of
$\lambda_p^{(1)}$ and $\underline{\lambda}_p^{(1)}$
are consistent with the previous ones
and we clearly have
\begin{alignat*}{3}
&\lambda_p^{(m)} &&\leq \lambda_p^{(m+1)}\,,\\
&\underline{\lambda}_p^{(m)} &&\leq
\underline{\lambda}_p^{(m+1)} \,,\\
&\underline{\lambda}_p^{(m)} &&\leq \lambda_p^{(m)}\,.
\end{alignat*}
\begin{thm}
\label{thm:high}
If $1<p<\infty$, for every $m\geq 1$ we have
\begin{alignat*}{3}
&\lim_{s\to p^-} \underline{\lambda}_s^{(m)} &&=
\lim_{s\to p^-} \lambda_s^{(m)} &&=
\underline{\lambda}_p^{(m)} \,,\\
&\lim_{s\to p^+} \underline{\lambda}_s^{(m)} &&=
\lim_{s\to p^+} \lambda_s^{(m)} &&=
\lambda_p^{(m)} \,.
\end{alignat*}
\end{thm}
\begin{proof}
Taking into account Theorems~\ref{thm:gammal} 
and~\ref{thm:gammar}, the assertions follow from the results 
of~\cite{champion_depascale2007, degiovanni_marzocchi2014}.
Let us give some detail following the
approach of~\cite{degiovanni_marzocchi2014}.
\par
If we define $g_p:L^1_{loc}(\Omega)\rightarrow\R$ as
\begin{equation}
\label{eq:gV}
g_p(u) = 
\begin{cases}
\displaystyle{
\left(\int_\Omega |u|^p\,dx\right)^{1/p}}
&\qquad\text{if $u\in L^p(\Omega)$}\,,\\
\noalign{\medskip}
0
&\qquad\text{otherwise}\,,
\end{cases} 
\end{equation}
it is easily seen that $g_p$ is
$L^1_{loc}(\Omega)$-continuous on
\[
\left\{u\in L^1_{loc}(\Omega):\,\,
\underline{\mathcal{E}}_p(u)\leq b\right\}
\]
for any $b\in\R$.
\par
If we consider
\[
\widehat{M} = 
\left\{u\in L^1_{loc}(\Omega):\,\,
g_p(u)=1\right\}
\]
endowed with the $L^1_{loc}(\Omega)$-topology,
by~\cite[Corollary~3.3]{degiovanni_marzocchi2014} we have
\begin{alignat*}{3}
&\left(\lambda_p^{(m)}\right)^{1/p} &&= 
\inf\biggl\{\sup_{u\in K} \mathcal{E}_p(u):\,\,
\text{$K$ is a nonempty, compact}
\biggr. \\ \biggl. &&&\qquad\qquad\qquad\qquad\qquad
\text{and symmetric subset of $\widehat{M}$ with 
$i(K)\geq m$}\biggr\}\,,\\
&\left(\underline{\lambda}_p^{(m)}\right)^{1/p} &&= 
\inf\biggl\{\sup_{u\in K} \underline{\mathcal{E}}_p(u):\,\,
\text{$K$ is a nonempty, compact}
\biggr. \\ \biggl. &&&\qquad\qquad\qquad\qquad\qquad
\text{and symmetric subset of $\widehat{M}$ with 
$i(K)\geq m$}\biggr\}\,,
\end{alignat*}
(see also~\cite[Theorem~5.2]{degiovanni_marzocchi2014}).
Then the assertions follow from
Theorems~\ref{thm:gammal}, \ref{thm:gammar}
and~\cite[Corollary~4.4]{degiovanni_marzocchi2014}
(see also~\cite[Theorem~6.4]{degiovanni_marzocchi2014}).
\end{proof}
\begin{cor}
\label{cor:high}
Let $1<p<\infty$ and assume that $\Omega$ is connected.
Then we have
\[
\lim_{s\to p} \lambda_s^{(m)} =
\lambda_p^{(m)} 
\qquad\forall m\geq 1
\]
if and only if
\[
\lim_{s\to p^-} \lambda_s^{(1)} =
\lambda_p^{(1)} \,.
\]
\end{cor}
\begin{proof}
If
\[
\lim_{s\to p^-} \lambda_s^{(1)} =
\lambda_p^{(1)} \,,
\]
from Theorem~\ref{thm:char} we infer that
$W^{1,p_-}_0(\Omega)=W^{1,p}_0(\Omega)$,
whence
$\underline{\lambda}_p^{(m)} = \lambda_p^{(m)}$
for any $m\geq 1$.
Then the assertion follows from Theorem~\ref{thm:high}.
\par
The converse is obvious.
\end{proof}

%--------------------------------------------------------------------

\section{Appendix}
\label{sect:app}
In this appendix we see that several well known properties of
$W^{1,p}_0(\Omega)$ are still valid for $W^{1,p_-}_0(\Omega)$.
\begin{thm}
\label{thm:mon}
If $1<p<\infty$, the following facts hold:
\begin{itemize}
\item[$(a)$]
for every $z\in L^p(\Omega)$ and $Z\in L^p(\Omega;\R^N)$,
there exists one and only one
$w\in W^{1,p_-}_0(\Omega)\subseteq W^{1,p}(\Omega)$ such that
\[
\int_\Omega |\nabla w|^{p-2}\nabla w\cdot\nabla v\,dx 
= \int_\Omega (zv+Z\cdot\nabla v)\,dx
\qquad\forall v\in W^{1,p_-}_0(\Omega)
\]
and the map 
\[
\begin{array}{rclcc}
L^p(\Omega)&\hskip-10pt \times &\hskip-10pt L^p(\Omega;\R^N)
&\rightarrow &W^{1,p}(\Omega) \\
\noalign{\medskip}
(z &\hskip-10pt , &\hskip-10pt Z) &\mapsto &w
\end{array}
\]
is continuous;
\item[$(b)$]
if $z_1, z_2\in L^p(\Omega)$ with $z_1\leq z_2$ a.e. in $\Omega$ 
and $w_1, w_2 \in W^{1,p_-}_0(\Omega)$ are the solutions of
\[
\int_\Omega |\nabla w_k|^{p-2}\nabla w_k\cdot\nabla v\,dx 
= \int_\Omega z_kv\,dx
\qquad\forall v\in W^{1,p_-}_0(\Omega)\,,
\]
then it holds $w_1\leq w_2$ a.e. in $\Omega$.
\end{itemize}
\end{thm}
\begin{proof}
Assertion $(a)$ easily follows from Proposition~\ref{prop:wp-}.
Since $(w_1-w_2)^+ \in W^{1,p_-}_0(\Omega)$, we have
\begin{alignat*}{3}
&\int_\Omega |\nabla w_1|^{p-2}\nabla w_1\cdot
\nabla (w_1-w_2)^+\,dx
= \int_\Omega  z_1(w_1-w_2)^+\,dx\,,\\
&\int_\Omega |\nabla w_2|^{p-2}\nabla w_2\cdot
\nabla (w_1-w_2)^+\,dx
= \int_\Omega z_2(w_1-w_2)^+\,dx\,,
\end{alignat*}
hence
\begin{multline*}
0\leq 
\int_{\{w_1 > w_2\}} \left(|\nabla w_1|^{p-2}\nabla w_1
- |\nabla w_2|^{p-2}\nabla w_2\right)\cdot
(\nabla w_1 - \nabla w_2)\,dx \\
= \int_\Omega (z_1-z_2)(w_1-w_2)^+\,dx \leq 0\,.
\end{multline*}
It follows $w_1\leq w_2$ a.e. in $\Omega$.
\end{proof}
\begin{lem}
\label{lem:pos}
If $\lambda\in\R$ and 
$u\in W^{1,p_-}_0(\Omega)\setminus\{0\}$ satisfy
\[
\int_\Omega |\nabla u|^{p-2}\nabla u\cdot\nabla v\,dx
= \lambda
\int_\Omega |u|^{p-2}u v\,dx
\qquad\forall v\in W^{1,p_-}_0(\Omega) \,,
\] 
then $u\in L^\infty(\Omega)\cap C^1(\Omega)$.
\par
Moreover, if $\Omega$ is connected and $u\geq 0$ a.e. in $\Omega$,
it holds $u>0$ in $\Omega$.
\end{lem}
\begin{proof}
If $p>N$ we have $W^{1,p_-}_0(\Omega)=W^{1,p}_0(\Omega)$
and the assertion is proved in~\cite{lindqvist1990}.
Therefore assume that $1<p\leq N$.
If we set
\[
R_k(t) = 
\begin{cases}
t+k &\qquad\text{if $t<-k$}\,,\\
\noalign{\medskip}
0 &\qquad\text{if $-k\leq t \leq k$}\,,\\
\noalign{\medskip}
t-k &\qquad\text{if $t>k$}\,,
\end{cases}
\]
we have $R_k(u)\in W^{1,p_-}_0(\Omega)$, hence
\[
\int_\Omega |\nabla R_k(u)|^p\,dx =
\lambda
\int_\Omega |u|^{p-1} |R_k(u)|\,dx\,.
\]
Let $1<s<p$ with $s^*\geq p$.
If we set
\[
A_k = \left\{x\in\Omega:\,\,|u(x)|>k\right\}=
\left\{x\in\Omega:\,\,R_k(u(x))\neq 0\right\}\,,
\]
it follows
\begin{alignat*}{3}
&\int_{A_k} (|u|-k)^p\,dx
&&\leq
\mathcal{L}^N(A_k)^{1 - \frac{p}{s^*}}
\left(\int_\Omega |R_k(u)|^{s^*}\,dx\right)^{\frac{p}{s^*}}\\
&&&\leq
c(N,s)^p\,\mathcal{L}^N(A_k)^{1 - \frac{p}{s^*}}
\left(\int_\Omega |\nabla R_k(u)|^{s}\,dx\right)^{\frac{p}{s}}\\
&&&\leq
c(N,s)^p\,\mathcal{L}^N(A_k)^{\frac{p}{N}}
\int_\Omega |\nabla R_k(u)|^{p}\,dx\\
&&&=
c(N,s)^p\,\mathcal{L}^N(A_k)^{\frac{p}{N}}\,
\lambda
\int_{A_k} |u|^{p-1} (|u|-k)\,dx\,.
\end{alignat*}
Then the same argument 
of~\cite[Lemma~4.1]{lindqvist1993}
shows that $u\in L^\infty(\Omega)$.
By the results of~\cite{dibenedetto1983}, \cite{tolksdorf1984},
we infer that $u\in C^1(\Omega)$.
\par
If $\Omega$ is connected and $u\geq 0$ a.e. in $\Omega$,
by~\cite[Theorem~5]{vazquez1984} we conclude that
$u>0$ in $\Omega$.
\end{proof}
\noindent
\emph{Proof of Theorem~\ref{thm:lambda1}.}
\par\noindent
If $u\in W^{1,p_-}_0(\Omega)\setminus\{0\}$ satisfies
\[
\int_\Omega |\nabla u|^p\,dx
= \underline{\lambda}_p^{(1)}\int_\Omega |u|^p\,dx \,,
\]
we claim that $u\in L^\infty(\Omega)\cap C^1(\Omega)$ and
either $u>0$ in $\Omega$ or $u<0$ in $\Omega$.
\par
Actually, by the minimality of $\underline{\lambda}_p^{(1)}$ 
it follows that
\[
\int_\Omega |\nabla u|^{p-2}\nabla u\cdot\nabla v\,dx
= \underline{\lambda}_p^{(1)}
\int_\Omega |u|^{p-2}u v\,dx
 \qquad\forall v\in W^{1,p_-}_0(\Omega)
\]
and from Lemma~\ref{lem:pos} we infer that
$u\in L^\infty(\Omega)\cap C^1(\Omega)$.
Moreover, also $w=|u|$ has the same properties and, 
by Lemma~\ref{lem:pos}, satisfies $w>0$ in $\Omega$.
Since $\Omega$ is connected, we have either $u=w$ or
$u=-w$ and the claim is proved.
\par
In particular, there exists 
$w\in W^{1,p_-}_0(\Omega)\cap L^\infty(\Omega)
\cap C^1(\Omega)$ such that $w>0$ in~$\Omega$ and
\[
\int_\Omega |\nabla w|^{p-2}
\nabla w\cdot\nabla v\,dx
= \underline{\lambda}_p^{(1)}
\int_\Omega w^{p-1} v\,dx
 \qquad\forall v\in W^{1,p_-}_0(\Omega) \,.
\]
\indent
Now let $\lambda\in\R$ and 
$u\in W^{1,p_-}_0(\Omega)\setminus\{0\}$ satisfy
\[
\begin{cases}
\text{$u\geq 0$ a.e. in $\Omega$}\,,\\
\noalign{\medskip}
\displaystyle{
\int_\Omega |\nabla u|^{p-2}\nabla u\cdot\nabla v\,dx
= \lambda \int_\Omega u^{p-1} v\,dx}
 \qquad\forall v\in W^{1,p_-}_0(\Omega) \,.
\end{cases}
\] 
Again, from Lemma~\ref{lem:pos} we infer that
$u\in L^\infty(\Omega)\cap C^1(\Omega)$ with $u>0$ in $\Omega$,
so that $\tilde{u}=\log u$ and $\tilde w=\log w$
also belong to $C^1(\Omega)$.
If we set $u_k=u+(1/k)$, $w_k=w+(1/k)$,
$\tilde{u}_k=\log u_k$ and $\tilde w_k=\log w_k$,
we have
\[
\frac{1}{u_k^{p-1}}\,\left(u_k^p - w_k^p\right)\,,\quad
\frac{1}{w_k^{p-1}}\,\left(w_k^p - u_k^p\right) 
\in W^{1,p_-}_0(\Omega)\,.
\]
The first function can be used as a test in
the equation of $u$ and the second one in that of $w$.
As in~\cite[Lemma~3.1]{lindqvist1990}, it follows
\begin{multline*}
\int_\Omega
\left(\lambda\,\frac{u^{p-1}}{u_k^{p-1}} - 
\underline{\lambda}_p^{(1)}\,\frac{w^{p-1}}{w_k^{p-1}}
\right)\left(
u_k^p - w_k^p\right)\,dx \\
=
\int_\Omega u_k^p\left[
|\nabla \tilde{u}_k|^p - |\nabla \tilde{w}_k|^p -
p|\nabla \tilde{w}_k|^{p-2}\nabla \tilde{w}_k\cdot
(\nabla \tilde{u}_k-\nabla \tilde{w}_k)\right]\,dx \\
+
\int_\Omega w_k^p\left[
|\nabla \tilde{w}_k|^p - |\nabla \tilde{u}_k|^p -
p|\nabla \tilde{u}_k|^{p-2}\nabla \tilde{u}_k\cdot
(\nabla \tilde{w}_k-\nabla \tilde{u}_k)\right]\,dx \geq 0\,.
\end{multline*}
Passing to limit as $k\to\infty$ and applying
Lebesgue's theorem and Fatou's lemma, we infer that
\begin{multline*}
(\lambda - \underline{\lambda}_p^{(1)})
\int_\Omega(u^p-w^p)\,dx \\
\geq
\int_\Omega u^p\left[
|\nabla \tilde{u}|^p - |\nabla \tilde{w}|^p -
p|\nabla \tilde{w}|^{p-2}\nabla \tilde{w}\cdot
(\nabla \tilde{u}-\nabla \tilde{w})\right]\,dx\\
+
\int_\Omega w^p\left[
|\nabla \tilde{w}|^p - |\nabla \tilde{u}|^p -
p|\nabla \tilde{u}|^{p-2}\nabla \tilde{u}\cdot
(\nabla \tilde{w}-\nabla \tilde{u})\right]\,dx \geq 0\,.
\end{multline*}
Since $u$ can be replaced by $tu$ for any $t>0$,
it follows $\lambda = \underline{\lambda}_p^{(1)}$.
Then the strict convexity of
$\left\{\xi\mapsto |\xi|^p\right\}$ implies that
$\nabla(\tilde{u}-\tilde{w})=0$ in $\Omega$.
Since $\Omega$ is connected, we infer that
$\tilde{u}=\tilde{w}+c$, hence
$u=e^c\,w$.
\par
On the other hand, if 
$u\in W^{1,p_-}_0(\Omega)\setminus\{0\}$ satisfies
\[
\int_\Omega |\nabla u|^{p-2}\nabla u\cdot\nabla v\,dx
= 	\underline{\lambda}_p^{(1)}
\int_\Omega |u|^{p-2}u v\,dx
 \qquad\forall v\in W^{1,p_-}_0(\Omega)\,,
\]
it follows
\[
\int_\Omega |\nabla u|^p\,dx
= \underline{\lambda}_p^{(1)}\int_\Omega |u|^p\,dx \,,
\]
hence $u\in C^1(\Omega)$ with either $u>0$ in $\Omega$
or $u<0$ in $\Omega$.
We infer that $u=t w$ for some $t\neq 0$.
\qed

%--------------------------------------------------------------------

%
\end{document}